\documentclass[a4paper,12pt]{amsart}
\usepackage{amsmath,amsxtra}
\usepackage{amssymb}
\usepackage{mathdots}
\usepackage[all]{xy}
\usepackage[dvipdfmx,%
bookmarks=true,%
bookmarksnumbered=true,%
setpagesize=false,%
linkcolor=blue,%
pdftitle={On conformal pseudo-subriemannian fundamental graded Lie algebras of semisimple type},%
pdfauthor={Tomoaki Yatsui},%
pdfsubject={graded Lie algebra},%
pdfkeywords={TeX; dvipdfmx; hyperref; color;}
]{hyperref}
\theoremstyle{plain}
\newtheorem{theorem}{Theorem}[section]
\newtheorem{lemma}{Lemma}[section]
\newtheorem{proposition}{Proposition}[section]

\theoremstyle{definition} 
\newtheorem{remark}{Remark}[section] 
 
\newtheorem{example}{Example}[section] 
\newtheorem{corollary}{Corollary}[section] 
\newcommand{\gl}{\mathfrak{gl}}

\newcommand{\ad}{\operatorname{ad}} 
\newcommand{\rea}{\operatorname{Re}}

\newcommand{\ima}{\operatorname{Im}} 
\newcommand{\cupprod}{\cup}
\newcommand{\capprod}{\cap}
\newcommand{\comp}{\circ} 
\newcommand{\rank}{\operatorname{rank}}

\newcommand{\Aut}{\operatorname{Aut}}
\newcommand{\gla}[1]{\mathfrak#1=\bigoplus\limits_{p\in\mathbb Z}\mathfrak#1_p}

\renewcommand{\labelenumi}{(\arabic{enumi})\,}

\makeatletter
\newcommand\ackname{Acknowledgements}
\if@titlepage
  \newenvironment{acknowledgements}{%
      \titlepage
      \null\vfil
      \@beginparpenalty\@lowpenalty
      \begin{center}%
        \bfseries \ackname
        \@endparpenalty\@M
      \end{center}}%
     {\par\vfil\null\endtitlepage}
\else
  
\fi
\makeatother
    \makeatletter
    \def\@citex[#1]#2{\leavevmode
      \let\@citea\@empty
      \@cite{\bfseries\@for\@citeb:=#2\do
        {\@citea\def\@citea{,\penalty\@m\ }%
         \edef\@citeb{\expandafter\@firstofone\@citeb\@empty}%
         \if@filesw\immediate\write\@auxout{\string\citation{\@citeb}}\fi
         \@ifundefined{b@\@citeb}{\hbox{\reset@font\bfseries ?}%
           \G@refundefinedtrue
           \@latex@warning
             {Citation `\@citeb' on page \thepage \space undefined}}%
           {\@cite@ofmt{\csname b@\@citeb\endcsname}}}}{{\rm#1}}}
    \makeatother
\title{Conformal pseudo-subriemannian fundamental graded Lie algebras of semisimple type}

\author{Tomoaki Yatsui}

\address{Masakae 1-9-2, Otaru, 047-0003, Japan}
\email{yatsui@frontier.hokudai.ac.jp}
\begin{document}
\maketitle

\begin{abstract}
We introduce the notion of a conformal pseudo-subriemannian fundamental graded 
Lie algebra of semisimple type. 
Moreover we give a classification of conformal pseudo-subriemannian 
fundamental graded Lie algebras of semisimple type and their prolongations. 
\end{abstract}

\section{Introduction and notation}
This paper is the sequel to the previous one \cite{yat15:01}. 
We first recall the notion of fundamental graded Lie algebras. 
Moreover we 
define the notion of conformal pseudo-subriemannian 
fundamental graded Lie algebras, which is a generalization of 
conformal subriemannian fundamental graded Lie algebras. 

A graded Lie algebra (GLA) $\mathfrak m=\bigoplus\limits_{p<0}\mathfrak g_p$ 
is called a fundamental graded Lie algebra (FGLA) 
if it is a finite dimensional graded Lie algebra generated by nonzero subspace 
$\mathfrak g_{-1}$. An FGLA $\mathfrak m$ is said to be of the $\mu$-th kind if 
$\mathfrak g_{-\mu}\ne\{0\}$ and $\mathfrak g_p=\{0\}$ for $p<-\mu$. 
Let $\mathfrak m=\bigoplus\limits_{p<0}\mathfrak g_p$ be 
an FGLA over $\mathbb R$ such that $\mathfrak g_{-2}\ne\{0\}$, and let $[g]$ be the conformal class of a nondegenerate symmetric bilinear form 
$g$ on $\mathfrak g_{-1}$. 
Then the pair $(\mathfrak m,[g])$ is called a conformal pseudo-subriemannian 
FGLA. 
In particular
if $g$ is positive definite, then $(\mathfrak m,[g])$ is 
called a conformal subriemannian FGLA. 
Also if the signature of $g$ has the form $(r,r)$, then 
$(\mathfrak m,[g])$ is called a conformal neutral-subriemannian 
FGLA. 
Note that if $(\mathfrak m,[g])$ is a conformal pseudo-subriemannian 
FGLA, so is $(\mathfrak m,[-g])$. 
Given two conformal pseudo-subriemannian FGLAs $(\mathfrak m_1,[g_1])$ and 
$(\mathfrak m_2,[g_2])$ we say that $(\mathfrak m_1,[g_1])$ 
is isomorphic to $(\mathfrak m_2,[g_2])$ if 
there exists a graded Lie algebra isomorphism $\varphi$ of $\mathfrak m_1$ 
onto $\mathfrak m_2$ such that $[\varphi^*g_2]=[g_1]$. 
Also we say that $(\mathfrak m_1,[g_1])$ 
is equivalent to $(\mathfrak m_2,[g_2])$ if 
$(\mathfrak m_1,[g_1])$ is isomorphic to $(\mathfrak m_2,[g_2])$ 
or $(\mathfrak m_2,[-g_2])$.  

Let $(\mathfrak m,[g])$ be a conformal pseudo-subriemannian FGLA, 
and let $\mathfrak g_0$ be the Lie algebra consisting of all the derivations 
$D$ of $\mathfrak m$ satisfying the following conditions: 
(1) $D(\mathfrak g_p)\subset \mathfrak g_p$ for all $p<0$; 
(2) $D|\mathfrak g_{-1}\in\mathfrak{co}(\mathfrak g_{-1},g)$. 
There exists 
a GLA $\gla l$ such that:
(i) $\mathfrak g_p=\mathfrak l_p$ for $p\leqq0$; 
(ii) $\gla l$ is transitive, i.e., for $X\in\mathfrak l_p$, $p\geqq0$, if $[X,\mathfrak l_{-1}]=\{0\}$, 
then $X=0$; (iii) $\gla l$ is maximum among GLAs satisfying 
conditions (i) and (ii) above, which is called the prolongation of 
$(\mathfrak m,[g])$ (For more details on the prolongation, 
see \cite[\S5]{tan70:1}). 
Note that the prolongation of $(\mathfrak m,[g])$ is finite dimensional 
(Lemma \ref{lem32}). 
Clearly the prolongation of $(\mathfrak m,[g])$ coincides with 
that of $(\mathfrak m,[-g])$. 

It is known (\cite{co15:01}, \cite{yat15:01}) that 
the prolongation $\gla g$ of a conformal subriemannian FGLA 
$(\mathfrak m,[g])$ satisfying the condition $\mathfrak g_1\ne\{0\}$ is 
a real rank one simple graded Lie algebra . 
In contrast, there exists a conformal neutral-subriemannian FGLA 
$(\mathfrak m,[g])$ such that the prolongation $\gla g$ of 
$(\mathfrak m,[g])$ is nonsemisimple and such that $\mathfrak g_1\ne\{0\}$ 
(cf. Example \ref{exa51}). 
A conformal pseudo-subriemannian FGLA is said to be of semisimple type 
if the prolongation is semisimple. 
In this paper we give a classification of 
conformal pseudo-subriemannian FGLAs of semisimple type 
and their prolongations (Theorem \ref{thm52}). 
In particular we prove that the prolongation of 
a conformal pseudo-subriemannian FGLA of semisimple type is simple.  
Also we give a classification of conformal neutral-subriemannian FGLAs of 
semisimple type (Corollary \ref{cor51}). 
\subsection*{Notation and conventions}
\begin{enumerate}
\item 
Blackboard bold is used for the standard systems 
$\mathbb Z$ (the ring of integers), $\mathbb R$ (real numbers),
$\mathbb C$ (complex numbers), $\mathbb C'$ (split complex numbers),
the real division rings $\mathbb H$ (Hamilton's quaternions), 
$\mathbb H'$ (split quaternions), 
$\mathbb O$ (Cayley's [nonassociative] octonions) 
and $\mathbb O'$ (split octonions). 
We denote by $\mathbb R_{>0}$ (resp. $\mathbb R_{\geqq0}$) 
the set consisting of all the positive real numbers 
(resp. non-negative real numbers). 
For $\mathbb K=\mathbb C$, $\mathbb C'$, $\mathbb H$, $\mathbb H'$, 
$\mathbb O$ or $\mathbb O'$, we set 
$\ima \mathbb K=\{~z\in\mathbb K:\rea z=0~\}$. 
  
\item For any real vector space $V$ 
we denote by $V(\mathbb C)$ the complexification of $V$. 

\item Let $V$ be a finite dimensional real vector space, 
and let $g$ be a nondegenerate symmetric bilinear form on $V$. We set 
$$\begin{aligned}
&\mathfrak{so}(V,g)=\{\,A\in \mathfrak{gl}(V):A\cdot g=0\,\}, \\
&\mathfrak{co}(V,g)=\{\,A\in \mathfrak{gl}(V):A\cdot g=\eta_Ag \ 
\text{for some} \ \eta_A\in\mathbb R\,\}, \\
\end{aligned}$$ 
where $A\cdot g$ is a symmetric bilinear form on $V$ defined 
by 
$(A\cdot g)(x,y)=g(Ax,y)+g(x,Ay)$ $(x,y\in V)$. 
We define a linear mapping $g^\flat:V\to V^*$ by 
$g^\flat(x)(y)=g(x,y)$ $(x,y\in V)$. 
Since $g$ is non-degenerate, $g^\flat$ is a linear isomorphism. 
We denote by $g^\sharp$ the inverse mapping of $g^\flat$. 

\item For a graded vector space $V=\bigoplus\limits_{p\in\mathbb Z}V_p$ and $k\in\mathbb Z$ 
we denote subspaces $\bigoplus\limits_{p\leqq k}V_p$ and 
$\bigoplus\limits_{p\geqq k}V_p$ by 
$V_{\leqq k}$ and $V_{\geqq k}$ respectively. 
Also 
we denote the subspace $\bigoplus\limits_{p<0}V_p$ by $V_-$. 
We call $V_-$ the negative part of $V$. 

\item 
For a reductive Lie algebra $\mathfrak g$, 
we denote by $\mathfrak g^{ss}$ the semisimple part of $\mathfrak g$.  

\item For a GLA $\gla g$ we denote by $\Aut_0(\mathfrak g)$ 
the group consisting of all the automorphisms $a$ of $\mathfrak g$ such that 
$a(\mathfrak g_p)=\mathfrak g_p$ for all $p\in\mathbb Z$. 
\end{enumerate}
\section{Finite dimensional semisimple graded Lie algebras}
\subsection{Finite dimensional complex semisimple graded Lie algebras}
Let $\gla g$ be a complex semisimple GLA 
such that the negative part $\mathfrak g_-$ is an FGLA. 
Let $\mathfrak h$ be a Cartan subalgebra of $\mathfrak g_0$; then $\mathfrak h$ is a Cartan subalgebra of $\mathfrak g$ such that $E\in \mathfrak h$, where 
$E$ is the characteristic element of $\gla g$ (i.e., 
$E$ is an element of $\mathfrak g_0$ such that 
$[E,X]=pX$ for $X\in\mathfrak g_p$). 
Let $\Delta$ be 
a root system of $(\mathfrak g,\mathfrak h)$. For $\alpha\in\Delta$, 
we denote by $\mathfrak g^{\alpha}$ the root space corresponding to 
$\alpha$. 
We associate to any set of roots $Q\subset \Delta$ a subspace 
$$\mathfrak g(Q)=\sum_{\alpha\in Q}\mathfrak g^\alpha\subset\mathfrak g.$$
There exists a simple root system $\Pi=\{\alpha_1,\dots,\alpha_l\}$ 
of $(\mathfrak g,\mathfrak h)$ 
such that $\mathfrak g(\Pi)\subset\bigoplus\limits_{p\geqq0}\mathfrak g_p$ 
(\cite[p. 441]{yam93:1}). 
Clearly $\alpha_i(E)$ is a non-negative integer. 
Since the negative part $\mathfrak g_-$ is generated by $\mathfrak g_{-1}$, 
$\alpha_i(E)$ is 0 or 1 (\cite[Lemma 3.8]{yam93:1}). 
We put $\Delta_p=\{\,\alpha\in\Delta: \alpha(E)=p\,\}$ and 
$\Pi_p=\Delta_p\capprod \Pi$; then $\Pi=\Pi_0\cupprod \Pi_1$. 
When $\gla g$ is a simple graded Lie algebra (SGLA), 
we enumerate simple roots of $\mathfrak g$ as in \cite{bou68:1}. 
Moreover if $\mathfrak g$ has the Dynkin diagram of type $X_l$, 
then $\gla g$ 
is said to be of type $(X_l,\Pi_1)$. 
\par
For $\gamma\in \Pi_1$, we put 
$$
\Delta_{-1}(-\gamma)=\{~-\gamma+(\Delta_0\cupprod \{0\})~\}\capprod \Delta
=\{~\alpha=-\gamma+\beta\in\Delta:\beta\in \Delta_0\cupprod \{0\}~\}.
$$
\begin{proposition}[{\cite[Ch.3, \S3.5]{gov94:0} and 
\cite[Proposition 7.3]{amt06:01}}] 
The decomposition of the $\mathfrak g_0$-module $\mathfrak g_{-1}$ 
into irreducible submodules is given by 
$$\mathfrak g_{-1}=\bigoplus_{\gamma\in \Pi_1}
\mathfrak g(\Delta_{-1}(-\gamma)).$$
In particular the 
$\mathfrak g_0$-module $\mathfrak g_{-1}$ is completely reducible. 
Moreover $\mathfrak g(\Delta_{-1}(-\gamma))$ is an irreducible $\mathfrak g_0$-module 
with highest weight $-\gamma$. 
\end{proposition}
From \cite[Ch.VIII, \S7, Propositions 11 and 12]{bou75:1} and the table of \cite{bou68:1} 
we obtain the following proposition. 
\begin{proposition} 
Let $\gla g$ be a finite dimensional complex SGLA satisfying the following conditions:
{\rm (i)} the negative part $\mathfrak m$ is an FGLA; 
{\rm (ii)} $\mathfrak g_{-2}$ and the semisimple part $\mathfrak g^{ss}_0$ of $\mathfrak g_0$ are both nonzero; 
{\rm (iii)} there exists a $\mathfrak g^{ss}_0$-invariant nondegenerate symmetric bilinear form $g:\mathfrak g_{-1}\times \mathfrak g_{-1}\to\mathbb C$. 
\begin{enumerate}
\item If the $\mathfrak g_0$-module $\mathfrak g_{-1}$ is irreducible, 
then $\gla g$ is of type 
$(C_l,\{\alpha_2\})$ $(l\geqq3)$ or $(F_4,\{\alpha_4\})$. 
\item If the $\mathfrak g_0$-module $\mathfrak g_{-1}$ is reducible and if 
$\mathfrak g_{-1}$ is the direct sum of two irreducible $\mathfrak g_0$-submodules of $\mathfrak g_{-1}$ which are totally isotropic with respect to $g$, 
then 
$\gla g$ is of type $(A_l,\{\alpha_1,\alpha_l\})$ or 
$(B_l,\{\alpha_1,\alpha_l\})$ $(l\geqq3)$. 
\end{enumerate}
\label{prop22}
\end{proposition}
\begin{remark}
Let $\gla g$ be a complex SGLA of type $(A_2,\{\alpha_1,\alpha_2\})$, 
$(B_2,\{\alpha_1,\alpha_2\})$ or $(G_2,\{\alpha_1,\alpha_2\})$. 
Then the semisimple part of $\mathfrak g_0$ is $\{0\}$. 
We can easily construct a nondegenerate symmetric bilinear form $g$ on $\mathfrak g_{-1}$ satisfying the following condition: 
for any $A\in \mathfrak g_0$ there exists 
a $\lambda_A\in \mathbb C$ 
such that 
$$g([A,X],Y)+g(X,[A,Y])=\lambda_Ag(X,Y)
\quad \text{for all}\  X,Y\in\mathfrak g_{-1}$$ 
(cf. Examples \ref{exa41}, 
\ref{exa42}, \ref{exa44}). 
\label{rem21}
\end{remark}
\subsection{Finite dimensional real semisimple graded Lie algebras}
In this subsection we describe gradations of finite dimensional real 
semisimple GLAs. We first notice the following proposition. 
\begin{proposition}[{\cite[Proposition 3.3]{yam93:1}}]
The finite dimensional real SGLAs $\gla g$ fall into the following 
two distinct classes:
\begin{enumerate}
\renewcommand{\labelenumi}{(\Alph{enumi}) }
\item The complex SGLAs, considered as real Lie algebras; 
\item The real form of complex simple Lie algebra so that 
$\mathfrak g(\mathbb C)
=\bigoplus\limits_{p\in\mathbb Z}\mathfrak g_p(\mathbb C)$ 
is a complex SGLA.
\end{enumerate}
\label{prop23}
\end{proposition}
Let $\gla g$ be a finite dimensional real semisimple GLA 
such that the negative part $\mathfrak g_-$ is an FGLA. 
Let $E$ be the characteristic element of $\gla g$, and 
$\mathfrak a$ a maximal $\mathbb R$-diagonalizable commutative 
subalgebra of $\mathfrak g$ containing $E$. 
Clearly $\mathfrak a$ is contained in $\mathfrak g_0$. 
There exists a Cartan decomposition $\mathfrak g=\mathfrak k\oplus \mathfrak p$ such that $\mathfrak a\subset \mathfrak p$ (\cite[Proposition 4.1]{gov94:0}). 
Let $\mathfrak h$ be a Cartan subalgebra of $\mathfrak g$ containing $\mathfrak a$. 
The complexification $\mathfrak h(\mathbb C)$ of $\mathfrak h$ 
is a Cartan subalgebra of $\mathfrak g(\mathbb C)$. 
Let $\Delta$ be the root system of 
$(\mathfrak g(\mathbb C),\mathfrak h(\mathbb C))$. 
We set 
$$\begin{aligned}
& \Delta_k=\{~\alpha\in \Delta:\alpha(E)=k~\}\quad (k\in\mathbb Z), \\
& \Delta^\bullet=\{~\alpha\in \Delta:\alpha(\mathfrak a)=\{0\}~\}, 
\quad \Delta^\circ=\Delta\setminus \Delta^\bullet. \\
\end{aligned}$$
Let $\sigma$ be the conjugation of $\mathfrak g(\mathbb C)$ 
defined by its real form $\mathfrak g$. 
For $\lambda\in\mathfrak h(\mathbb C)^*$ we define 
the element $\lambda^\sigma\in \mathfrak h(\mathbb C)^*$ 
by $\lambda^\sigma=\overline{\lambda\comp\sigma}$.  
If $\alpha\in \Delta$, then $\alpha^\sigma\in\Delta$. 
We can choose a simple root system $\Pi$ of 
$(\mathfrak g(\mathbb C),\mathfrak h(\mathbb C))$ such that: (i)  
the corresponding system of positive roots $\Delta^+$ satisfies the following 
conditions: 
$\Delta^+\cap \Delta^\circ$ is $\sigma$-invariant; (ii)  
$\mathfrak g(\Pi)\subset \mathfrak g(\mathbb C)_{\geqq0}$. 
There exists an involutive permutation $\nu$ of the set 
$\Pi^\circ$ such that 
$$\gamma^\sigma=\nu(\gamma)+\sum_{\beta\in \Pi^\bullet}k_\beta\beta
\quad (\gamma\in \Pi^\circ,k_\beta\in\mathbb Z_{\geqq0}).$$
We set 
$\Pi^\bullet= \Delta^\bullet\cap \Pi$, 
$\Pi^\circ=\Delta^\circ\cap \Pi$ and 
$\Pi_k=\Delta_k\cap\Pi$. 
We shall identify the vertices of the Dynkin diagram $X_l$ 
with the elements of $\Pi$. 
The Satake diagram $S_l$ is obtained from $X_l$ as follows: 
Firstly we paint the vertices $\alpha\in \Pi^\bullet$ 
(resp. $\alpha\in \Pi^\circ$) into black (resp. white). 
Secondly for $\alpha\in \Pi^\circ$, 
if $\alpha^\sigma\ne \alpha$, then 
we connect the pair $\{\alpha,\alpha^\sigma\}$ by a curved arrow. 
When this is done for all such pairs, 
we obtain the Satake diagram $S_l$.  

Let $\gla g$ be a finite dimensional real semisimple GLA 
with Satake diagram $S_l$, and let 
$\Delta_k$, $\Pi$ and $\Pi_k$ be as in the above. 
Since $\mathfrak g_-$ is an FGLA, 
$\Pi=\Pi_0\cup \Pi_1$. 
Furthermore the following properties hold: 
(i) $\Pi^\bullet\subset \Pi_0$; 
(ii) $\Pi_1\subset \Pi^\circ$; 
(iii) If $\alpha\in \Pi_1$, then $\alpha^\sigma\in\Delta_1$ 
(\cite[Theorem 8.1]{amt06:01}). 
The semisimple GLA $\gla g$ is said to be of type 
$(S_l,\Pi_1)$ (\cite[\S2]{doj82:0} and \cite[\S3.4]{yam93:1}). 
For simplicity we denote by $\mathfrak g^{\mathbb C}_{-1}(-\gamma)$ 
the subspace $\mathfrak g(\mathbb C)(\Delta_{-1}(-\gamma))$ of 
$\mathfrak g_{-1}(\mathbb C)$, where $\gamma\in\Pi_1$. 
\begin{proposition}[{\cite[Proposition 8.3]{amt06:01}}] 
Let $\gla g$ be a finite dimensional real semisimple GLA of type 
$(S_l,\Pi_1)$. 
For $\gamma\in \Pi_1$, there are two possibilities: 
\begin{enumerate}
\item $\nu(\gamma)=\gamma$. 
Then $-\gamma^\sigma\in\Delta_{-1}(-\gamma)$ 
and the $\mathfrak g_0(\mathbb C)$-module 
$\mathfrak g^{\mathbb C}_{-1}(-\gamma)$ 
is $\sigma$-invariant. 
\item $\nu(\gamma)\ne\gamma$. 
Then $-\gamma^\sigma\in\Delta_{-1}(-\nu(\gamma))$ 
and the two irreducible $\mathfrak g_0(\mathbb C)$-modules 
$\mathfrak g^{\mathbb C}_{-1}(-\gamma)$ and 
$\mathfrak g^{\mathbb C}_{-1}(-\nu(\gamma))$
determine one irreducible $\mathfrak g_0$-submodule 
$\mathfrak g\capprod (\mathfrak g^{\mathbb C}_{-1}(-\gamma)
+\mathfrak g^{\mathbb C}_{-1}(-\nu(\gamma)))$ of $\mathfrak g_{-1}$. 
\end{enumerate}
\label{prop24}
\end{proposition}
\section{Conformal pseudo-subriemannian fundamental graded Lie algebras} 
Let $\mathfrak m=\bigoplus\limits_{p<0}\mathfrak g_p$ be an FGLA 
of the $\mu$-kind over $\mathbb R$, where $\mu\geqq2$. 
Let $g_1$ and $g_2$ be two nondegenerate real symmetric bilinear forms on 
$\mathfrak g_{-1}$. 
We say that 
$g_1$ is equivalent to $g_2$ if 
there exists an $\eta\in\mathbb R_{>0}$ such that $g_2 = \eta g_1$. 
We denote by $[g]$ the equivalence class of a nondegenerate real symmetric bilinear form $g$ on 
$\mathfrak g_{-1}$, which is called the conformal class of $g$. 

Let $g$ be a nondegenerate real symmetric bilinear form on 
$\mathfrak g_{-1}$ with signature $(r,s)$. 
We call the pair $(\mathfrak m,[g])$ a conformal pseudo-subriemannian FGLA of type $(r,s)$. 
In particular, if $s=0$ (resp. $r=s$), then 
$(\mathfrak m,[g])$ is called a conformal subriemannian FGLA 
(resp. a conformal neutral-subriemannian FGLA). 
 
Let $(\mathfrak m,[g])$ be a conformal pseudo-subriemannian FGLA, and let 
$\mathfrak g_0$ be the Lie algebra consisting of all the derivations $D$ of 
$\mathfrak m$ satisfying 
the following conditions (i) and (ii): 
(i) $D(\mathfrak g_p)\subset \mathfrak g_p$ for all $p<0$; 
(ii) $D|\mathfrak g_{-1}\in\mathfrak{co}(\mathfrak g_{-1},g)$. 
Let $\gla g$ be the prolongation of $(\mathfrak m,\mathfrak g_0)$ 
(see \cite[\S5.2]{tan70:1}). 
We call the transitive GLA $\gla g$ the prolongation of 
$(\mathfrak m,[g])$. 
If $\mathfrak g$ is finite dimensional and semisimple, then 
$(\mathfrak m,[g])$ is said to be of semisimple type. 

Let $(\mathfrak m_1,[g_1])$ and $(\mathfrak m_2,[g_2])$ be two 
conformal pseudo-subriemannian FGLAs. 
We say that $(\mathfrak m_1,[g_1])$ is isomorphic to $(\mathfrak m_2,[g_2])$ 
if there exists a graded Lie algebra isomorphism $\varphi$ of $\mathfrak m_1$ onto $\mathfrak m_2$ such that $[\varphi^*g_2]=[g_1]$. 
Also we say that $(\mathfrak m_1,[g_1])$ is equivalent to 
$(\mathfrak m_2,[g_2])$ if $(\mathfrak m_1,[g_1])$ is isomorphic to 
$(\mathfrak m_2,[g_2])$ or $(\mathfrak m_2,[-g_2])$.  

The following lemma can be proved by the same methods as in 
the case of conformal subriemannian FGLAs (\cite[Lemma 3.1]{yat15:01}). 
\begin{lemma}
Let $(\mathfrak m,[g])$ be a conformal pseudo-subriemannian FGLA, 
and let $\gla g$ be the prolongation of $(\mathfrak m,[g])$. 
Let $\rho_{-1}$ be the representation of $\mathfrak g_0$ on $\mathfrak g_{-1}$ 
defined by $\rho_{-1}(A)(X)=[A,X]$ 
$(A\in\mathfrak g_0,X\in \mathfrak g_{-1})$. 
We set 
$\hat{\mathfrak g}_0=(\rho_{-1})^{-1}(\mathfrak{so}(\mathfrak g_{-1},g))$. 
Then 
\begin{enumerate}
\item $[\mathfrak g_0,\mathfrak g_0]\subset \hat{\mathfrak g}_0$. 
\item Let $E$ be the characteristic element of $\gla g$. 
Then $\mathfrak g_0=\mathbb RE\oplus \hat{\mathfrak g}_0$. 
\end{enumerate}
\label{lem31}
\end{lemma}

\begin{lemma}
Let $(\mathfrak m,[g])$ be a conformal pseudo-subriemannian FGLA, and 
let $\gla g$ be the prolongation of $(\mathfrak m,[g])$. 
Then  $\mathfrak g$ is finite dimensional. 
\label{lem32}
\end{lemma}
\begin{proof}
We first assume that $\dim\mathfrak g_{-1}\geqq3$. 
We define a subalgebra $\mathfrak h_0$ of $\mathfrak g_0$ as follows:
$$\mathfrak h_0=\{\,X\in\mathfrak g_0:[X,\mathfrak g_{\leqq-2}]=\{0\}\}.$$
Identifying $\mathfrak h_0$ with a subalgebra of 
$\mathfrak{gl}(\mathfrak g_{-1})$, we see that 
$\mathfrak h_0\subset \mathfrak{co}(\mathfrak g_{-1},g)$. 
Since the second algebraic prolongation 
$\mathfrak{co}(\mathfrak g_{-1},g)^{(2)}$ of $\mathfrak{co}(\mathfrak g_{-1},g)$ vanishes, 
we get $\mathfrak h_0^{(2)}=\{0\}$. 
From Corollary 1 of Theorem 11.1 in \cite{tan70:1}, it follows that 
$\mathfrak g$ is finite dimensional. 
Next we assume that $\dim\mathfrak g_{-1}=2$. 
There exists a basis $(e_1,e_2)$ of $\mathfrak g_{-1}$ such that 
$g(e_i,e_j)=\varepsilon_i\delta_{ij}$ for all $i,j=1,2$, where 
$\varepsilon_i\in\{-1,1\}$. 
Note that $[e_1,e_2]\ne0$. 
For $A\in\mathfrak h_0$, 
we set 
$\ad A(e_i)=\sum\limits_{k=1}^2a_{ki}e_k$ $(i=1,2;a_{ki}\in\mathbb R)$. 
Since $g([A,e_i],e_j)+g(e_1,[A,e_2])=\lambda_Ag(e_i,e_j)$, 
we see that $2a_{ii}=\lambda_A$ and $a_{ji}\varepsilon_j+a_{ij}\varepsilon_i=0$. Also since $[A,[e_i,e_j]]=0$, we get 
$a_{11}+a_{22}=0$, so $\lambda_A=0$. 
Hence $\mathfrak h_0$ is considered as a subalgebra of 
$\mathfrak{so}(\mathfrak g_{-1},g)$. 
However since the first algebraic prolongation 
$\mathfrak{so}(\mathfrak g_{-1},g)^{(1)}$ of 
$\mathfrak{so}(\mathfrak g_{-1},g)$ 
vanishes, we see that $\mathfrak g$ is finite dimensional. 
\end{proof}

\begin{lemma}
Let $(\mathfrak m,[g])$ be a conformal pseudo-subriemannian FGLA, and let 
$\gla g$ be the prolongation of $(\mathfrak m,[g])$.   
If $\gla l$ is a transitive semisimple GLA such that 
$\mathfrak g_p=\mathfrak l_p$ for all $p<0$ and 
$\ad(\mathfrak l_0)|\mathfrak g_{-1}\subset\mathfrak{co}(\mathfrak g_{-1},g)$, 
then $\mathfrak g$ coincides with $\mathfrak l$. 
\label{lem33}
\end{lemma}
\begin{proof}
Since $\ad(\mathfrak l_0)|\mathfrak g_{-1}\subset \mathfrak{co}(\mathfrak g_{-1},g)$, 
$\gla l$ is a graded subalgebra of $\gla g$. 
Let $\mathfrak r$ be the radical of $\mathfrak g$; then 
$\mathfrak r$ is a graded ideal of $\gla g$: 
$\gla r$, $\mathfrak r_p=\mathfrak r\capprod \mathfrak g_p$. 
Since $\mathfrak m=\mathfrak l_-$, we see that $\mathfrak r_-=\{0\}$. 
By transitivity of $\gla g$, we get $\mathfrak r=\{0\}$, so 
$\mathfrak g$ is semisimple. 
Since $\dim \mathfrak g_p=\dim \mathfrak g_{-p}=\dim \mathfrak l_{-p}
=\dim \mathfrak l_p$ for $p>0$, we get 
$\mathfrak g_p=\mathfrak l_p$ for $p\ne0$. 
Since $\mathfrak l_0=[\mathfrak l_{-1},\mathfrak l_{1}]
=[\mathfrak g_{-1},\mathfrak g_{1}]=\mathfrak g_0$, 
we obtain $\mathfrak g=\mathfrak l$. 
\end{proof}
The following lemma is essentially due to the proof of \cite[Lemma 4.1]{BA02:1}. 
\begin{lemma}
Let $(\mathfrak m,[g])$ be a conformal pseudo-subriemannian FGLA of type 
$(r,s)$, and let 
$\gla g$ be the prolongation of $(\mathfrak m,[g])$.   
If $\mathfrak a$ is a maximal $\mathbb R$-diagonalizable commutative subalgebra of $\mathfrak g$ contained in $\mathfrak g_0$,  
then $\dim \mathfrak a\leqq \min\{r,s\}+1$. 
In particular, if $\mathfrak g$ is semisimple, then we have 
$\rank_\mathbb R\mathfrak g\leqq\min\{r,s\}+1$. 
\end{lemma}
\begin{proof} 
Clearly $\mathfrak a$ contains the characteristic element $E$ of 
$\gla g$. By lemma \ref{lem31}, 
$\mathfrak a$ can be decomposed into the direct sum 
$\mathfrak a'\oplus \mathbb RE$, 
where $\mathfrak a'$ is a subalgebra of $\mathfrak a$ such that 
$\ad(\mathfrak a')|\mathfrak g_{-1}\subset \mathfrak{so}(\mathfrak g_{-1},g)$. 
Then $\mathfrak a'$ is $\mathbb R$-diagonalizable in $\mathfrak g_{-1}$. 
Let $\lambda,\mu$ be weights of the $\mathfrak a'$-module $\mathfrak g_{-1}$ 
and let $V^\lambda,V^\mu$ be the corresponding weight spaces. 
For $x\in V^\lambda$, $y\in V^\mu$ and $t\in \mathfrak a'$, we get 
$$0=g([t,x],y)+g(x,[t,y])=(\lambda+\mu)(t)g(x,y).$$
Hence if $\lambda+\mu\ne0$, then $g(V^\lambda,V^\mu)=0$. 
Let $\hat{\mathfrak a}$ be the subspace of $\mathfrak a'^*$ 
spanned by the weights of the $\mathfrak a'$-module $\mathfrak g_{-1}$. 
Since the $\mathfrak a'$-module $\mathfrak g_{-1}$ is faithful, 
the annihilator space $\{\,h\in\mathfrak a':\lambda(h)=0\quad 
\text{for all}\ \lambda\in \hat{\mathfrak a}\,\}$ vanishes, so 
$\dim \hat{\mathfrak a}=\dim \mathfrak a'$. Thus 
the weights of the module span $\mathfrak a'^*$. 
There exists a basis $(\lambda_1,\dots,\lambda_l)$ of $\mathfrak a'^*$ 
such that each $\lambda_i$ is a weight of the $\mathfrak a'$-module $\mathfrak g_{-1}$. Then 
$U=\bigoplus\limits_{i=1}^lV^{\lambda_i}$ is a totally isotropic subspace of 
$(\mathfrak g_{-1},g)$, so 
$\dim \mathfrak a-1= \dim U\leqq \min\{r,s\}$. 
If $\mathfrak g$ is semisimple, then $\rank_\mathbb R\mathfrak g$  
equals to $\dim\mathfrak a$, so we obtain 
$\rank_\mathbb R\mathfrak g\leqq\min\{r,s\}+1$. 
\end{proof}

\section{Examples of conformal pseudo-subriemannian FGLAs of semisimple type}
\subsection{Conformal pseudo-subriemannian FGLAs of classical type}
\label{sec41}
\begin{example}[{cf. \cite[\S9]{tak65:01} and \cite[Example 3.1.2, p.241]{cs09:0}}] 
Let $\mathbb K$ be $\mathbb C$, $\mathbb H$, 
$\mathbb C'$ or $\mathbb H'$. 
Here we consider $\mathbb K$ as an $\mathbb R$-algebra. 
We put 
$\mathfrak l=
\mathfrak{sl}(n,\mathbb K)$ $(n\geqq3)$; 
then $\mathfrak l$ is a real simple Lie algebra. 
Let $K_m$ be the $m\times m$ matrix whose $(i,j)$-component is 
$\delta_{i,m+1-j}$. 
We define an $n\times n$ symmetric real matrix $S_{p,q}$ as follows: 
$$S_{p,q}=
\begin{bmatrix}
0 & 0 & K_p \\
0 & 1_q & 0 \\
K_p & 0 & 0 \\
\end{bmatrix}\qquad (p\geqq1,q\geqq0,2p+q=n\geqq3),$$
where $1_{q}$ denotes the $q\times q$ identity matrix. 
Here the center column and the center row of $S_{p,q}$ should be deleted 
when $q=0$. 
Then $S_{p,q}$ is a symmetric real matrix with signature $(p+q,p)$ 
such that $S_{p,q}^2=S_{p,q}$. 
We put 
$
\mathfrak{g}=\{\,X\in\mathfrak l:
X^*S_{p,q}+S_{p,q}X=O\,\}$; then 
$$\mathfrak{g}=\left\{
X=\begin{bmatrix}
X_{11} & X_{12} & X_{13} \\
X_{21} & X_{22} & -S_{p-1,q}X^*_{12} \\
X_{31} & -X^*_{21}S_{p-1,q} & -\overline{X_{11}} \\
\end{bmatrix}\in\mathfrak l
:
\begin{aligned} 
& X_{11}\in \mathbb K,\ X_{12}\in M(1,n',\mathbb K), \\ 
& X_{21}\in M(n',1,\mathbb K), \\ 
&X_{31},X_{13}\in \ima \mathbb K, 
X_{22}\in \mathfrak{gl}(n',\mathbb K), \\
&X_{22}+S_{p-1,q}X^*_{22}S_{p-1,q}=O
\end{aligned}\right\},
$$
where $n'=n-2$ and we set $S_{0,m}=1_m$. 
Here $M(p,q,\mathbb K)$ denotes the set of $\mathbb K$-valued $p\times q$-matrices. 
We define subspaces $\mathfrak g_p$ of $\mathfrak g$ as follows: 
$$\begin{aligned}
&\mathfrak g_{-2}
 =\left\{\begin{bmatrix}
0 & 0 & 0 \\
0 & 0 & 0 \\
X_{31} & 0 & 0 \\
\end{bmatrix}\in\mathfrak g:
X_{31}\in\ima \mathbb K
\right\}, \\
& 
\mathfrak g_{-1}=\left\{\begin{bmatrix}
0 & 0 & 0 \\
X_{21} & 0 & 0 \\
0 & -X^*_{21}S_{p-1,q} & 0 \\
\end{bmatrix}\in\mathfrak g: 
X_{21}\in M(n',1,\mathbb K)\right\}, \\
& 
\mathfrak g_{0}=\left\{\begin{bmatrix}
X_{11} & 0 & 0 \\
0 & X_{22} & 0 \\
0 & 0 & -\overline{X_{11}} \\
\end{bmatrix}\in\mathfrak g: 
\begin{aligned} 
& X_{11}\in \mathbb K, 
X_{22}\in \mathfrak{gl}(n',\mathbb K), \\
& X_{22}+S_{p-1,q}X^*_{22}S_{p-1,q}=O
\end{aligned}
\right\}, \\
& 
\mathfrak g_p=\{~X\in\mathfrak g:
{}^tX\in\mathfrak g_{-p}~\}\quad (p=1,2), 
\quad \mathfrak g_p=\{0\} \quad (|p|>2). 
\\
\end{aligned}$$
Then $\gla g$ becomes a GLA whose negative part $\mathfrak m$ is an FGLA of the second kind. 
We define a symmetric bilinear form $g$ on $\mathfrak g_{-1}$ as follows:  
$$g(X, Y)=\rea(X^*_{21}S_{p-1,q}Y_{21}),\qquad X,Y\in \mathfrak g_{-1}.$$
Then $g$ is nondegenerate and 
for $A\in\mathfrak g_0$ we obtain 
$(\ad(A)|\mathfrak g_{-1})\cdot g=-2(\rea A_{11})g$. 
Hence 
$\ad(\mathfrak g_0)|\mathfrak g_{-1}\subset\mathfrak{co}(\mathfrak g_{-1},g)$. 
The conformal pseudo-subriemannian FGLA $(\mathfrak m,[g])$ is said to be 
of type $({\rm H}\mathbb K)_{p,q}$.  

In case $\mathbb K=\mathbb C$ or $\mathbb C'$  
we know that $\mathfrak g$ is denoted by 
$\mathfrak{su}(p+q,p,\mathbb K)$ $(n=2p+q)$. 
Note that $\mathfrak{su}(p+q,p,\mathbb C')$ is isomorphic to 
$\mathfrak{sl}(2p+q,\mathbb R)$ for any $p,q$. 
If $\mathbb K=\mathbb C$ (resp. $\mathbb K=\mathbb C'$), then 
$\gla g$ is a real SGLA of type 
$(\textrm{(AIIIa)}_{l,p},\{\alpha_1,\alpha_l\})$ $(l=n-1=2p+q-1, 
p\geqq2,q\geqq1)$, 
$(\textrm{(AIIIb)}_{l},\{\alpha_1,\alpha_l\})$ $(l=n-1=2p-1, p\geqq2,q=0)$ or
$(\textrm{(AIV)}_{l},\{\alpha_1,\alpha_l\})$ $(l=n-1=q+1,p=1,q\geqq1)$ 
(resp. $(\textrm{(AI)}_l,\{\alpha_1,\alpha_l\})$ $(l=n-1=2p+q-1\geqq2)$), 
and $g$ has the signature $(2p+2q-2,2p-2)=(2l-2p,2p-2)$ 
(resp. $(2p+q-2,2p+q-2)=(l-1,l-1)$).  
Here 
$\textrm{(AIIIa)}_{l,p}$, 
$\textrm{(AIIIb)}_{l}$,  
$\textrm{(AIV)}_{l}$ and 
$\textrm{(AI)}_l$ are the following Satake diagrams.  

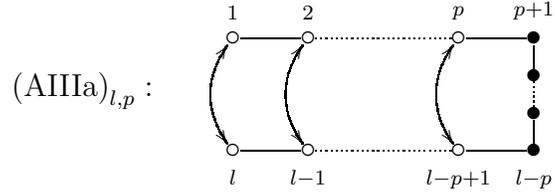
\begin{center}
\begin{minipage}{5cm}
\begin{xy}
(-20,-7) *{\textrm{(AIIIa)}_{l,p}:}="K", 
(0,0) *{\circ}="A"*++!D{{\scriptstyle 1}},
(10,0) *{\circ}="B"*++!D{{\scriptstyle 2}},
(30,0) *{\circ}="C"*++!D{{\scriptstyle p}},
(40,0) *{\bullet}="D"*++!D{{\scriptstyle p+1}},
(40,-5) *{\bullet}="E",
(40,-10) *{\bullet}="F",
(40,-15) *{\bullet}="G"*++!U{{\scriptstyle l-p}},
(30,-15) *{\circ}="H"*++!U{{\scriptstyle l-p+1}},
(10,-15) *{\circ}="I"*++!U{{\scriptstyle l-1}},
(0,-15) *{\circ}="J"*++!U{{\scriptstyle l}},
\ar @{-} "A";"B"
\ar @{.} "B";"C"
\ar @{-} "C";"D"
\ar @{-} "D";"E"
\ar @{.} "E";"F"
\ar @{-} "F";"G"
\ar @{-} "G";"H"
\ar @{.} "H";"I"
\ar @{-} "I";"J"
\ar @{<->} @/_3mm/ "A";"J"
\ar @{<->} @/_3mm/ "B";"I"
\ar @{<->} @/_3mm/ "C";"H"
\end{xy}
\end{minipage}
\end{center}

\begin{center}
\begin{minipage}{5cm}
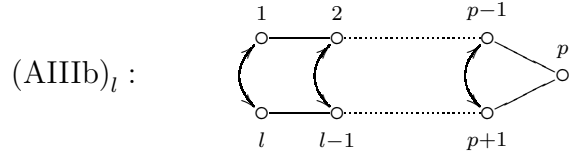
\begin{xy}
(-25,-5) *{\textrm{(AIIIb)}_{l}:}="K", 
(0,0) *{\circ}="A"*++!D{{\scriptstyle 1}},
(10,0) *{\circ}="B"*++!D{{\scriptstyle 2}},
(30,0) *{\circ}="C"*++!D{{\scriptstyle p-1}},
(40,-5) *{\circ}="D"*++!D{{\scriptstyle p}},
(30,-10) *{\circ}="E"*++!U{{\scriptstyle p+1}},
(10,-10) *{\circ}="F"*++!U{{\scriptstyle l-1}},
(0,-10) *{\circ}="G"*++!U{{\scriptstyle l}},
\ar @{-} "A";"B"
\ar @{.} "B";"C"
\ar @{-} "C";"D"
\ar @{-} "D";"E"
\ar @{.} "E";"F"
\ar @{-} "F";"G"
\ar @{<->} @/_3mm/ "A";"G"
\ar @{<->} @/_3mm/ "B";"F"
\ar @{<->} @/_3mm/ "C";"E"
\end{xy}
\end{minipage}
\end{center}

\begin{center}
\begin{minipage}{5cm}
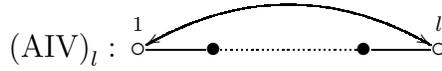
\begin{xy}
(-10,0) *{\textrm{(AIV)}_l:}="K", 
(0,0) *{\circ}="A"*++!D{{\scriptstyle 1}},(10,0) *{\bullet}="B",(30,0) *{\bullet}="C",
(40,0) *{\circ}="D"*++!D{{\scriptstyle l}},
\ar @{-} "A";"B"
\ar @{.} "B";"C"
\ar @{-} "C";"D"
\ar @{<->} @/^6mm/ "A";"D"
\end{xy}
\end{minipage}
\end{center}

\begin{center}
\begin{minipage}{5cm}
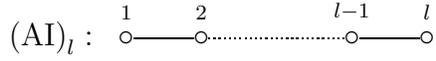
\begin{xy}
(-10,0) *{\textrm{(AI)}_{l}:}="K", 
(0,0) *{\circ}="A"*++!D{{\scriptstyle 1}},
(10,0) *{\circ}="B"*++!D{{\scriptstyle 2}},
(30,0) *{\circ}="C"*++!D{{\scriptstyle l-1}},
(40,0) *{\circ}="D"*++!D{{\scriptstyle l}},
\ar @{-} "A";"B"
\ar @{.} "B";"C"
\ar @{-} "C";"D"
\end{xy}
\end{minipage}
\end{center}

In case $\mathbb K=\mathbb H$ or $\mathbb H'$  
we know that $\mathfrak g$ is denoted by $\mathfrak{sp}(p+q,p,\mathbb K)$. 
Note that $\mathfrak{sp}(p+q,p,\mathbb H')$ is isomorphic to 
$\mathfrak{sp}(2p+q,\mathbb R)$ for any $p,q$ and that 
$\mathfrak{sp}(n,\mathbb R)$ is denoted by $\mathfrak{sp}(2n,\mathbb R)$ or 
$\mathfrak{sp}_{2n}(\mathbb R)$ in 
\cite{cs09:0}, \cite{gov94:0} and \cite{0ni04:0}. 
If $\mathbb K=\mathbb H$ (resp. $\mathbb K=\mathbb H'$), then 
$\gla g$ is a real SGLA of type 
$(\textrm{(CIIa)}_{l,p},\{\alpha_2\})$ $(l=n=2p+q\geqq3,p,q\geqq1)$, or 
$(\textrm{(CIIb)}_{l},\{\alpha_2\})$ ($n=l=2p\geqq3,q=0$) 
(resp. $(\textrm{(CI)}_l,\{\alpha_2\})$ $(l=n=2p+q\geqq3)$) 
and $g$ has the signature $(4p+4q-4,4p-4)=(4l-4p-4,4p-4)$ 
(resp. $(4p+2q-4,4p+2q-4)=(2l-4,2l-4)$).  
Here 
$\textrm{(CIIa)}_{l,p}$, 
$\textrm{(CIIb)}_{l}$ and 
$\textrm{(CI)}_l$ are the following Satake diagrams.  
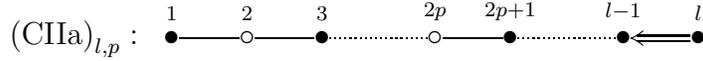
\begin{center}
\begin{minipage}{5cm}
\begin{xy}
(-13,0) *{\textrm{(CIIa)}_{l,p}:}="K", 
(0,0) *{\bullet}="A"*++!D{{\scriptstyle 1}},
(10,0) *{\circ}="B"*++!D{{\scriptstyle 2}},
(20,0) *{\bullet}="C"*++!D{{\scriptstyle 3}},
(35,0) *{\circ}="D"*++!D{{\scriptstyle 2p}},
(45,0) *{\bullet}="E"*++!D{{\scriptstyle 2p+1}},
(60,0) *{\bullet}="F"*++!D{{\scriptstyle l-1}},
(70,0) *{\bullet}="G"*++!D{{\scriptstyle l}},
\ar @{-} "A";"B"
\ar @{-} "B";"C"
\ar @{.} "C";"D"
\ar @{-} "D";"E"
\ar @{.} "E";"F"
\ar @{=>} "G";"F"
\end{xy}
\end{minipage}
\end{center}

\begin{center}
\begin{minipage}{5cm}
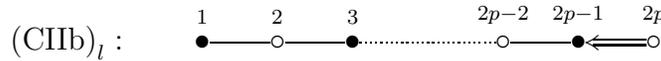
\begin{xy}
(-18,0) *{\textrm{(CIIb)}_{l}:}="K", 
(0,0) *{\bullet}="A"*++!D{{\scriptstyle 1}},
(10,0) *{\circ}="B"*++!D{{\scriptstyle 2}},
(20,0) *{\bullet}="C"*++!D{{\scriptstyle 3}},
(40,0) *{\circ}="D"*++!D{{\scriptstyle 2p-2}},
(50,0) *{\bullet}="E"*++!D{{\scriptstyle 2p-1}},
(60,0) *{\circ}="F"*++!D{{\scriptstyle 2p}},
\ar @{-} "A";"B"
\ar @{-} "B";"C"
\ar @{.} "C";"D"
\ar @{-} "D";"E"
\ar @{=>} "F";"E"
\end{xy}
\end{minipage}
\end{center}

\begin{center}
\begin{minipage}{5cm}
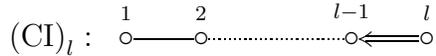
\begin{xy}
(-10,0) *{\textrm{(CI)}_{l}:}="K", 
(0,0) *{\circ}="A"*++!D{{\scriptstyle 1}},
(10,0) *{\circ}="B"*++!D{{\scriptstyle 2}},
(30,0) *{\circ}="C"*++!D{{\scriptstyle l-1}},
(40,0) *{\circ}="D"*++!D{{\scriptstyle l}},
\ar @{-} "A";"B"
\ar @{.} "B";"C"
\ar @{=>} "D";"C"
\end{xy}
\end{minipage}
\end{center}
By Lemma \ref{lem33}, the prolongation of a conformal 
pseudo-subriemannian FGLA of type $({\rm H}\mathbb K)_{p,q}$ 
coincides with $\gla g$. 
\label{exa41}
\end{example}

\begin{example}[cf.{\cite[\S4.4 (3)]{yam93:1}}] 
We put 
$\mathfrak g=
\{\,X\in \gl(2l+1,\mathbb R):
{}^tXS+SX=0\}$ $(l\geqq2)$, 
where $S=S_{l,1}$. 
More explicitly 
$$
\mathfrak{g}
=\left\{
X=\begin{bmatrix}
A & a & B \\
\xi & 0 & -a' \\
C & -\xi' & -A' \\
\end{bmatrix}
\in\mathfrak{gl}(2l+1,\mathbb R)
:
\begin{aligned} 
& A,B,C\in \mathfrak{gl}(l,\mathbb R), \\
& B=-B', C=-C', \\
& a\in M(l,1,\mathbb R), \xi\in M(1,l,\mathbb R) 
\end{aligned} 
\right\}.$$
Here for an $r\times s$-matrix $X$ we put $X'=K_s{}^t\!XK_r$. 
The Lie algebra $\mathfrak g$ is a real simple Lie algebra 
$\mathfrak{so}(l+1,l,\mathbb R)$ of type ${\rm(BI)}_{l,l}$. 
Here ${\rm (BI)}_{l,l}$ is the following Satake diagram: 
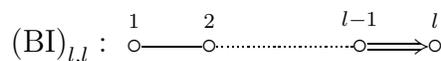
\begin{center}
\begin{minipage}{5cm}
\begin{xy}
(-10,0) *{\textrm{(BI)}_{l,l}:}="K", 
(0,0) *{\circ}="A"*++!D{{\scriptstyle 1}},
(10,0) *{\circ}="B"*++!D{{\scriptstyle 2}},
(30,0) *{\circ}="C"*++!D{{\scriptstyle l-1}},
(40,0) *{\circ}="D"*++!D{{\scriptstyle l}},
\ar @{-} "A";"B"
\ar @{.} "B";"C"
\ar @{=>} "C";"D"
\end{xy}
\end{minipage}
\end{center}

We define subspaces $\mathfrak g_p$ of $\mathfrak g$ as follows: 
$$\begin{aligned}
&\mathfrak g_{-3}
 =\left\{\begin{bmatrix}
0 & 0 & 0 & 0 & 0 \\
0 & 0 & 0 & 0 & 0 \\
0 & 0 & 0 & 0 & 0 \\
X_{41} & 0 & 0 & 0 & 0 \\
0 & -X'_{41} & 0 & 0 & 0 \\
\end{bmatrix}\in\mathfrak g:
X_{41}\in M(l-1,1,\mathbb R)
\right\}, \\
& 
\mathfrak g_{-2}
=\left\{\begin{bmatrix}
0 & 0 & 0 & 0 & 0 \\
0 & 0 & 0 & 0 & 0 \\
X_{31} & 0 & 0 & 0 & 0 \\
0 & X_{42} & 0 & 0 & 0 \\
0 & 0 & -X_{31} & 0 & 0 \\
\end{bmatrix}\in\mathfrak g: 
\begin{aligned}
& X_{31}\in\mathbb R,X_{42}\in\mathfrak{gl}(l-1,\mathbb R), \\
& X'_{42}=-X_{42} \\
\end{aligned} 
\right\}, \\
& 
\mathfrak g_{-1}
=\left\{\begin{bmatrix}
0 & 0 & 0 & 0 & 0 \\
X_{21} & 0 & 0 & 0 &0 \\
0 & X_{32} & 0 & 0 & 0  \\
0 & 0 & -X'_{32} & 0 & 0 \\
0 & 0 & 0 & -X'_{21} &0 \\
\end{bmatrix}\in\mathfrak g: 
\begin{aligned} 
& X_{21}\in M(l-1,1,\mathbb R), \\
& X_{32}\in M(1,l-1,\mathbb R) 
\end{aligned}
\right\}, \\
&
\mathfrak g_{0}=\left\{\begin{bmatrix}
X_{11} & 0 & 0 & 0 & 0 \\
0 & X_{22} & 0 & 0 & 0\\
0 & 0 & 0 & 0 & 0 \\
0 & 0 & 0 & -X'_{22} & 0   \\
0 & 0 &  0 & 0 & -X_{11}  \\
\end{bmatrix}\in\mathfrak g: 
X_{11}\in \mathbb R, 
X_{22}\in\mathfrak{gl}(l-1,\mathbb R) \right\}, \\
& 
\mathfrak g_p=\{~X\in\mathfrak g:
{}^tX\in\mathfrak g_{-p}~\}\quad (p=1,2,3), 
\quad \mathfrak g_p=\{0\} \quad (|p|>3). 
\\
\end{aligned}$$
Then $\gla g$ becomes a real SGLA of 
type $({\rm (BI)}_{l,l},\{\alpha_1,\alpha_l\})$ whose negative part $\mathfrak m$ is an FGLA of the third kind. 
We define a symmetric bilinear form $g$ on $\mathfrak g_{-1}$ by 
$$g(X, Y)=-\frac12(X_{32}Y_{21}+Y_{32}X_{21})\qquad(X,Y\in \mathfrak g_{-1}).$$ 
Then $g$ is nondegenerate, and for $A\in\mathfrak g_{0}$ we see that 
$(\ad(A)|\mathfrak g_{-1})\cdot g=-A_{11}g$. 
Hence $\ad(\mathfrak g_0)|\mathfrak g_{-1}\subset \mathfrak{co}(\mathfrak g_{-1},g)$. 
By Lemma \ref{lem33}
$(\mathfrak m,[g])$ is 
a conformal neutral-subriemannian FGLA 
such that $\gla g$ is the prolongation of 
$(\mathfrak m,[g])$. 
The conformal pseudo-subriemannian FGLA $(\mathfrak m,[g])$ is said to be 
of type $({\rm BI})_l$.  
\label{exa42}
\end{example}

\subsection{Conformal pseudo-subriemannian FGLAs of exceptional type}
\begin{example}[{\cite[\S3]{gom96:0}}]
Let $\mathbb K=\mathbb O$ or $\mathbb O'$. 
Here we consider $\mathbb K$ as an $\mathbb R$-algebra. 
We define a nondegenerate symmetric bilinear form $g$ on $\mathbb K$ by 
$g(x,y)=\frac12(\bar{x}y+\bar{y}x)$.  
We set 
$$\mathfrak g_{-1}=\mathfrak g_{1}=\mathbb K,\quad \mathfrak g_{-2}=\mathfrak g_{2}=\ima \mathbb K, \quad 
\mathfrak g_0=\mathfrak g'_0\oplus\mathbb R, 
\quad \mathfrak g_p=\{0\}\quad \text{for}\ |p|>2,$$ 
where $\mathfrak g'_0=\{~A\in\mathfrak{so}(\mathbb K,g):A(1)=0~\}$. 
Note that $\mathfrak g'_0$ is isomorphic to 
$\mathfrak{so}(\ima\mathbb K,g)$. 
We further put 
$\gla g$ and $\mathfrak m=\mathfrak g_{-2}\oplus\mathfrak g_{-1}$.   
We define a bracket operation $[\,\cdot,\cdot\,]$ on $\mathfrak g$ as in 
\cite[\S3.2, p.444 and \S3.4, pp.447--448]{gom96:0}. 
By using \cite[Lemma 3.1]{gom96:0} we can prove that $\gla g$ becomes a GLA whose negative part 
$\mathfrak m$ is an FGLA of the second kind. 
For $A\oplus r\in\mathfrak g_0$ and $X,Y\in\mathfrak g_{-1}$ 
we see that 
$(\ad(A\oplus r)|\mathfrak g_{-1})\cdot g=-2rg$, 
and hence $\ad(\mathfrak g_0)|\mathfrak g_{-1}\subset\mathfrak{co}(\mathfrak g_{-1},g)$. 
The conformal pseudo-subriemannian FGLA $(\mathfrak m,[g])$ is said to be 
of type $({\rm H}\mathbb K)$.  

By \cite[Theorem 3.5]{gom96:0}, 
in case $\mathbb K=\mathbb O$ (resp. $\mathbb K=\mathbb O'$) the GLA 
$\gla g$ is a real SGLA of type 
$({\rm FII},\{\alpha_4\})$ (resp. $({\rm FI},\{\alpha_4\})$). 
Here 
${\rm FII}(={\rm F_{4(-20)}})$ and ${\rm FI}(={\rm F_{4(4)}})$ are the 
following Satake diagrams respectively. 
\begin{center}
\begin{minipage}{5cm}
\begin{xy}
(-10,0) *{\textrm{FII}:}="K",
(0,0) *{\bullet}="A"*++!D{{\scriptstyle 1}},
(10,0) *{\bullet}="B"*++!D{{\scriptstyle 2}},
(20,0) *{\bullet}="C"*++!D{{\scriptstyle 3}},
(30,0) *{\circ}="D"*++!D{{\scriptstyle 4}},
\ar @{-} "A";"B"
\ar @{=>} "B";"C"
\ar @{-} "C";"D"
\end{xy}
\end{minipage}
\hspace{1cm}
\begin{minipage}{5cm}
\begin{xy}
(-10,0) *{\textrm{FI}:}="K",
(0,0) *{\circ}="A"*++!D{{\scriptstyle 1}},
(10,0) *{\circ}="B"*++!D{{\scriptstyle 2}},
(20,0) *{\circ}="C"*++!D{{\scriptstyle 3}},
(30,0) *{\circ}="D"*++!D{{\scriptstyle 4}}
\ar @{-} "A";"B"
\ar @{=>} "B";"C"
\ar @{-} "C";"D"
\end{xy}
\end{minipage}
\end{center}
Clearly $(\mathfrak m,[g])$ is a conformal subriemannian FGLA 
(resp. a conformal neutral-subriemannian FGLA) when $\mathbb K=\mathbb O$ (resp. $\mathbb K=\mathbb O'$). 
By Lemma \ref{lem33}, the prolongation of a conformal 
pseudo-subriemannian FGLA of type $({\rm H}\mathbb K)$ 
coincides with $\gla g$. 
\label{exa43}
\end{example}
\begin{example}
Let $V$ be a real vector space $\mathbb R^2$, 
and we set $\mathfrak s=\mathfrak{sl}(V)$. 
We define real vector spaces $\mathfrak l_p$ $(p\in\mathbb Z)$ as follows: 
$$\mathfrak l_{-2}=\mathfrak l_2=\mathbb R,\quad 
\mathfrak l_{-1}=\mathfrak l_1=S^3(V), \quad 
\mathfrak l_0=\mathfrak s\oplus\mathbb R,\quad 
\mathfrak l_p=\{0\}\quad  (p>|2|).$$
We define a bracket operation $[\cdot,\cdot]$ on $\gla l$ 
as in \cite[\S3.1, p.450]{gom96:0}. 
Then 
$\gla l$ becomes a real SGLA of type 
$({\rm G_{2(2)}},\{\alpha_2\})$  
(\cite[Theorem 4.3]{gom96:0}) and the negative part $\mathfrak m$ is an FGLA 
of the second kind. 
Here ${\rm G_{2(2)}}$ is the following Satake diagram: 
\begin{center}
\begin{minipage}{5cm}
\begin{xy}
(-10,0) *{\textrm{G}_{2(2)}:}="K",
(0,0) *{\circ}="A"*++!D{{\scriptstyle 1}},
(10,0) *{\circ}="B"*++!D{{\scriptstyle 2}},
\ar @3{->} "B";"A"
\end{xy}
\end{minipage}
\end{center}

We set $V_{-1}=\mathbb R e_1$ and $V_{-2}=\mathbb R e_2$, 
where $(e_1,e_2)$ is the canonical basis of $V$. 
We put 
$\mathfrak s_p=\{X\in\mathfrak s:X(V_k)\subset V_{k+p}\ \text{for all}\ k\}$; 
then $\gla s$ is a real SGLA of the first kind. 

We define subspaces $W_k$ (resp. $W_{-k}$) $(k=1,\dots,4)$ of 
$\mathfrak l_{1}$ (resp. $\mathfrak l_{-1})$ as follows: 
$$
W_{\pm1}=S^3(V_{-1}),\quad 
W_{\pm2}=S^2(V_{-1})\otimes V_{-2},\quad 
W_{\pm3}=V_{-1}\otimes S^2(V_{-2}),\quad 
W_{\pm4}=S^3(V_{-2});$$ 
then $\mathfrak l_{\pm1}=\bigoplus\limits_{k=1}^{4}W_{\pm k}$. 
We define subspaces $\mathfrak g_p$ of $\mathfrak l$ as follows: 
$$\begin{aligned}
& \mathfrak g_{\pm5}=\mathfrak l_{\pm2},\quad 
\mathfrak g_{k}=W_k\quad (k=\pm2,\pm3,\pm4), \quad 
\mathfrak g_{\pm1}=\mathfrak s_{\pm1}\oplus W_{\pm1}, \\ 
&\mathfrak g_0=\mathfrak s_0\oplus \mathbb R, \quad 
\mathfrak g_p
=\{0\}\quad  (p>|5|). 
\end{aligned}$$
Then $\mathfrak l=\gla g$ becomes a real SGLA of type 
$({\rm G}_{2(2)},\{\alpha_1,\alpha_2\})$ such that 
the negative part $\mathfrak m$ is an FGLA of the 5-th kind. 
Let $(\cdot\mid\cdot)$ be an inner product on $S^3(V)$ induced by 
the canonical inner product on $V$. 
We define a symmetric bilinear form $g$ on $\mathfrak g_{-1}$ as follows: 
$$g(\mathfrak s_{-1},\mathfrak s_{-1})=g(W_{-1},W_{-1})=0, 
\quad g(X,u)=g(u,X)=(Xu\mid e_1^2e_2)\quad (X\in\mathfrak s_{-1},u\in W_{-1}). $$
Then $g$ is nondegenerate, and for $A=\lambda(E_{11}-E_{22})\oplus r\in \mathfrak g_0$, $X\in \mathfrak s_{-1}$ and $u\in W_{-1}$ 
($\lambda,r\in\mathbb R)$,  
we see that 
$(\ad (A)|\mathfrak g_{-1})\cdot g=(\lambda -r)g$, 
where $E_{ij}$ is an element of $\mathfrak s$ such that $E_{ij}e_k=\delta_{jk}e_i$. 
Thus $\ad(\mathfrak g_0)|\mathfrak g_{-1}\subset \mathfrak{co}(\mathfrak g_{-1},g)$. 
Hence by Lemma \ref{lem33}
$(\mathfrak m,[g])$ is 
a conformal neutral-subriemannian FGLA 
such that $\gla g$ is the prolongation of 
$(\mathfrak m,[g])$. 
The conformal pseudo-subriemannian FGLA $(\mathfrak m,[g])$ is said to be 
of type $({\rm G})$.  

\label{exa44}
\end{example}

\section{Classification of conformal pseudo-subriemannian FGLAs of semisimple type}
In this section we prove that 
a conformal pseudo-subriemannian FGLA of semisimple type is 
isomorphic to one of conformal pseudo-subriemannian FGLAs given in 
the previous section. 
\begin{proposition}
Let $(\mathfrak m,[g])$ be a conformal pseudo-subriemannian FGLA of 
semisimple type, and let $\gla g$ be the prolongation of $(\mathfrak m,[g])$. 
\begin{enumerate}
\item If the $\mathfrak g_0$-module $\mathfrak g_{-1}$ 
is irreducible and 
the $\mathfrak g_0(\mathbb C)$-module $\mathfrak g_{-1}(\mathbb C)$ is 
reducible, there exist $\mathfrak g_0(\mathbb C)$-submodules 
$\mathfrak g_{-1}(\mathbb C)^{(i)}$ $(i=1,2)$ such that 
{\rm (i)}
$\mathfrak g_{-1}(\mathbb C)=\mathfrak g_{-1}(\mathbb C)^{(1)}\oplus
\mathfrak g_{-1}(\mathbb C)^{(2)}$; 
{\rm (ii)} 
$\mathfrak g_{-1}(\mathbb C)^{(i)}$ $(i=1,2)$ are totally isotropic subspaces 
of $(\mathfrak g_{-1}(\mathbb C),g)$; 
{\rm (iii)} 
$\mathfrak g_{-1}(\mathbb C)^{(1)}$ is contragredient to 
$\mathfrak g_{-1}(\mathbb C)^{(2)}$ as a 
$\hat{\mathfrak g}_0(\mathbb C)$-module, where 
$\hat{\mathfrak g}_0=(\rho_{-1})^{-1}(\mathfrak{so}(\mathfrak g_{-1},g))$. 
\item 
If the $\mathfrak g_0$-module $\mathfrak g_{-1}$ is reducible, 
then there exist 
$\mathfrak g_0$-submodules 
$\mathfrak g_{-1}^{(i)}$ $(i=1,2)$ such that: 
{\rm (i)} 
$\mathfrak g_{-1}=\mathfrak g_{-1}^{(1)}\oplus
\mathfrak g_{-1}^{(2)}$; 
{\rm (ii)} 
$\mathfrak g_{-1}^{(i)}$ $(i=1,2)$ are totally isotropic subspaces of 
$(\mathfrak g_{-1},g)$; 
{\rm (iii)} 
$\mathfrak g_{-1}^{(1)}$ is contragredient to 
$\mathfrak g_{-1}^{(2)}$ as a $\hat{\mathfrak g}_0$-module; 
{\rm (iv)} 
the $\mathfrak g_0(\mathbb C)$-modules $\mathfrak g_{-1}^{(i)}(\mathbb C)$ are 
irreducible. 
\item $\gla g$ is an SGLA of class {\rm(B)}. 
\end{enumerate}
\label{prop51}
\end{proposition}
\begin{proof}
(1) and (2). 
We decompose $\mathfrak g_{-1}$ (resp. $\mathfrak g_{-1}(\mathbb C)$) 
into irreducible $\mathfrak g_0$-modules 
(resp. $\mathfrak g_0(\mathbb C)$-modules) as follows: 
$$\mathfrak g_{-1}=\bigoplus_{i=1}^k\mathfrak g_{-1}^{(i)},\quad 
\mathfrak g_{-1}(\mathbb C)=\bigoplus_{i=1}^{k'}\mathfrak g_{-1}(\mathbb C)^{(i)}
.$$
Let $E_i$ be an element of $\mathfrak g_0$ such that 
$[E_i,X_j]=-\delta_{ij}X_j$ for all $X_j\in\mathfrak g_{-1}^{(j)}$. 
We first assume that $\mathfrak g_{-1}^{(1)}$ is a nondegenerate subspace of 
$(\mathfrak g_{-1},g)$. 
There exist elements $X_1,Y_1$ of $\mathfrak g_{-1}^{(1)}$ such that 
$g(X_1,Y_1)\ne0$. 
Since 
$$g([E_1,X_1],Y_1)+g(X_1,[E_1,Y_1])=\eta_{E_1}g(X_1,Y_1),$$ 
we see that $\eta_{E_1}=-2$. 
For $X_i\in\mathfrak g_{-1}^{(i)}$ $(i\geqq2)$, 
$$g([E_1,X_1],X_i)+g(X_1,[E_1,X_i])=\eta_{E_1}g(X_1,X_i),$$ 
so $g(X_1,X_i)=0$. 
Thus we get $g(\mathfrak g_{-1}^{(1)},\mathfrak g_{-1}^{(i)})=0$ $(i\geqq2)$. 
If there exists a $j\geqq2$ such that 
$\mathfrak g_{-1}^{(j)}$ is a nondegenerate subspace of 
$(\mathfrak g_{-1},g)$, then 
$$0=g([E_j,X_1],Y_1)+g(X_1,[E_j,Y_1])=\eta_{E_j}g(X_1,Y_1)=-2g(X_1,Y_1),$$ 
which is a contradiction. 
Hence $\mathfrak g_{-1}^{(i)}$ $(i\geqq2)$ are 
totally isotropic subspaces of $(\mathfrak g_{-1},g)$. 
Assume that $\mathfrak g_{-1}^{(2)}\ne\{0\}$. 
There exists $j\geqq3$ such that the restriction of $g$ to the space 
$\mathfrak g_{-1}^{(1)}\times \mathfrak g_{-1}^{(j)}$ is nondegenerate. 
We set $E'_2=E_2+E_j$. 
Let $X_2$ (resp. $X_j)$ be an element of $\mathfrak g_{-1}^{(2)}$ 
(resp. $\mathfrak g_{-1}^{(j)}$) such that $g(X_2,X_j)\ne0$. 
Since 
$$g([E'_2,X_2],X_j)+g(X_2,[E'_2,X_j])=\eta_{E'_2}g(X_2,X_j),$$ 
we see that $\eta_{E'_2}=-2$. 
Also 
$$0=g([E'_2,X_1],Y_1)+g(X_1,[E'_2,Y_1])=-2g(X_1,Y_1).$$ 
This is a contradiction. 
Therefore we obtain that $\mathfrak g_{-1}$ is an irreducible $\mathfrak g_0$-module. 
Next we assume that $\mathfrak g_{-1}^{(i)}$ is a totally isotropic subspace of $(\mathfrak g_{-1},g)$. 
Here we may assume that the restriction of $g$ to 
$\mathfrak g_{-1}^{(1)}\times \mathfrak g_{-1}^{(2)}$ is nondegenerate. 
From the above result, 
$\mathfrak g_{-1}^{(2)}$ is a totally isotropic subspace of 
$(\mathfrak g_{-1},g)$ and is contragredient to 
$\mathfrak g_{-1}^{(1)}$ as a $\hat{\mathfrak g}_0$-module. 
If the restriction of $g$ to $\mathfrak g_{-1}^{(1)}\times \mathfrak g_{-1}^{(3)}$ is nondegenerate, $\mathfrak g_{-1}^{(3)}$ 
is contragredient to 
$\mathfrak g_{-1}^{(1)}$ as a $\hat{\mathfrak g}_0$-module, 
so $\mathfrak g_{-1}^{(1)}$ is isomorphic to $\mathfrak g_{-1}^{(3)}$ as 
a $\mathfrak g_0$-module, which is a contradiction. 
Hence $g(\mathfrak g_{-1}^{(1)},\mathfrak g_{-1}^{(3)})=0$. 
Similarly we get $g(\mathfrak g_{-1}^{(2)},\mathfrak g_{-1}^{(3)})=0$. 
There exists $k\geqq4$ such that 
the restriction of $g$ to $\mathfrak g_{-1}^{(3)}\times \mathfrak g_{-1}^{(k)}$ is nondegenerate. 
We set $E'_1=E_1+E_2$. 
Let $X_i$ $(i=1,2,3,k)$ be elements of $\mathfrak g_{-1}^{(i)}$ 
such that $g(X_1,X_2)\ne0$ and $g(X_3,X_k)\ne0$. Since 
$$0=g([E'_1,X_1],X_2)+g(X_1,[E'_1,X_2])=\eta_{E'_1}g(X_1,X_2),$$ 
we get $\eta_{E'_1}=-2$. On the other hand, we see that 
$$0=g([E'_1,X_3],X_k)+g(X_3,[E'_1,X_k])=\eta_{E'_1}g(X_3,X_k),$$ 
which is a contradiction. 
Hence $\mathfrak g_{-1}=\mathfrak g_{-1}^{(1)}\oplus \mathfrak g_{-1}^{(2)}$. 
Similarly we can prove that 
if 
the $\mathfrak g_0(\mathbb C)$-module $\mathfrak g_{-1}(\mathbb C)$ is 
reducible, there exist $\mathfrak g_0(\mathbb C)$-submodules 
$\mathfrak g_{-1}(\mathbb C)^{(i)}$ $(i=1,2)$ such that 
(i)
$\mathfrak g_{-1}(\mathbb C)=\mathfrak g_{-1}(\mathbb C)^{(1)}\oplus
\mathfrak g_{-1}(\mathbb C)^{(2)}$; 
(ii) 
$\mathfrak g_{-1}(\mathbb C)^{(i)}$ $(i=1,2)$ are totally isotropic subspaces 
of $(\mathfrak g_{-1}(\mathbb C),g)$; (iii) 
$\mathfrak g_{-1}(\mathbb C)^{(1)}$ is contragredient to 
$\mathfrak g_{-1}(\mathbb C)^{(2)}$ as a 
$\hat{\mathfrak g}_0(\mathbb C)$-module. 
The assertions (1) and (2) follow from these results. 

(3) We assume that $\mathfrak g$ is not simple. 
There exist ideals $\mathfrak a^{(1)}$ and $\mathfrak a^{(2)}$ of 
$\mathfrak g$ such that $\mathfrak a^{(1)}$ is a simple ideal of 
$\mathfrak g$ and 
$\mathfrak g=\mathfrak a^{(1)}\oplus\mathfrak a^{(2)}$. 
Both ideals $\mathfrak a^{(i)}$ $(i=1,2)$ are graded ideals of $\mathfrak g$; 
we write $\mathfrak a^{(i)}=\bigoplus\limits_{p\in\mathbb Z}\mathfrak a_p^{(i)}$. 
By transitivity of $\gla g$, we see that $\mathfrak a_{-1}^{(i)}\ne\{0\}$ $(i=1,2)$. 
From the results of (1) and (2) 
$\mathfrak a_{-1}^{(2)}$ is contragredient to $\mathfrak a_{-1}^{(1)}$ as a 
$\hat{\mathfrak g}_0$-module, which is a contradiction. 
Hence $\mathfrak g$ is simple. Also 
from  the results of (1) and (2) and from 
\cite[p.157, Example 2]{gov94:0}, it follows that 
$\mathfrak g$ is of class (B). 
\end{proof}
We decompose the conformal pseudo-subriemannian FGLAs of semisimple type 
into the following three classes: 
\begin{enumerate}
\renewcommand{\labelenumi}{(S\Roman{enumi}) }
\item The $\mathfrak g_0(\mathbb C)$-module $\mathfrak g_{-1}(\mathbb C)$ 
is irreducible. 
\item The $\mathfrak g_0$-module $\mathfrak g_{-1}$ 
is irreducible and the $\mathfrak g_0(\mathbb C)$-module $\mathfrak g_{-1}(\mathbb C)$ 
is reducible. 
\item The $\mathfrak g_0$-module $\mathfrak g_{-1}$ is reducible. 
\end{enumerate}

\begin{theorem}
Let $(\mathfrak m,[g])$ be a conformal pseudo-subriemannian FGLA of semisimple type, 
and let $\gla g$ be the prolongation of $(\mathfrak m,[g])$. 
Assume that $(\mathfrak m,[g])$ is of type $(r,s)$ $(r\geqq s)$. 
\begin{enumerate}
\item If $(\mathfrak m,[g])$ is of class ${\rm(SI)}$, 
then $\gla g$ is an SGLA of type 
$(({\rm CI})_{l},\{\alpha_2\})$,  
$(({\rm CIIa})_{l,p},\{\alpha_2\})$, 
$(({\rm CIIb})_{l},\{\alpha_2\})$ $(l\geqq3,p\geqq1)$, 
$({\rm FI},\{\alpha_4\})$, 
or $({\rm FII},\{\alpha_{4}\})$.

\item If $(\mathfrak m,[g])$ is 
of class ${\rm (SII)}$, 
then $\gla g$ is an SGLA of type 
$(({\rm AIIIa})_{l,p},\{\alpha_1,\alpha_l\})$, 
$(({\rm AIIIb})_{l},\{\alpha_1,\alpha_l\})$ 
or 
$(({\rm AIV})_{l},\{\alpha_1,\alpha_l\})$ $(l\geqq2)$. 

\item If $(\mathfrak m,[g])$ is class ${\rm (SIII)}$, 
then $(\mathfrak m,[g])$ is conformal neutral-subriemannian and 
$\gla g$ is an SGLA of type $(({\rm AI})_{l},\{\alpha_1,\alpha_l\})$, $(({\rm BI})_{l,l},\{\alpha_1,\alpha_l\})$ $(l\geqq2)$ or 
$({\rm G_{2(2)}},\{\alpha_1,\alpha_2\})$. 

\end{enumerate}
\label{thm51}
\end{theorem}
\begin{proof} 
By Proposition \ref{prop51} the complexification 
$\mathfrak g(\mathbb C)
=\bigoplus\limits_{p\in\mathbb Z}\mathfrak g_p(\mathbb C)$ 
is an SGLA. 
We first assume that $(\mathfrak m,[g])$ is of class (SI); then 
$\mathfrak g(\mathbb C)
=\bigoplus\limits_{p\in\mathbb Z}\mathfrak g_p(\mathbb C)$ be of 
type $(X_l,\{\alpha_i\})$. Furthermore 
$\mathfrak g_{-1}(\mathbb C)$ is an irreducible 
$\mathfrak g_0(\mathbb C)$-module with highest weight $-\alpha_i$ and 
there exists a $\mathfrak g^{ss}_0(\mathbb C)$-invariant symmetric bilinear form $g$ on $\mathfrak g_{-1}(\mathbb C)$. 
By Proposition \ref{prop22} (1) 
we obtain that $(X_l,\{\alpha_i\})$ is one of 
$(C_{l},\{\alpha_2\})$ $(l\geqq3)$, 
$(F_4,\{\alpha_4\})$. 
Next we assume that $(\mathfrak m,[g])$ is of class (SII) or (SIII). 
By Proposition \ref{prop51} (2), the $\mathfrak g_0(\mathbb C)$-module 
$\mathfrak g_{-1}(\mathbb C)$ is decomposed as follows: 
$\mathfrak g_{-1}(\mathbb C)=\mathfrak g_{-1}(\mathbb C)^{(1)}\oplus 
\mathfrak g_{-1}(\mathbb C)^{(2)}$, 
where $\mathfrak g_{-1}(\mathbb C)^{(i)}$ ($i=1,2)$ are 
irreducible $\mathfrak g_{0}(\mathbb C)$-submodule of 
$\mathfrak g_{-1}(\mathbb C)$ such that: (i) 
each $\mathfrak g_{-1}(\mathbb C)^{(i)}$ is totally isotropic with respect to $g$; (ii) 
$\mathfrak g_{-1}(\mathbb C)^{(1)}$ is contragredient to 
$\mathfrak g_{-1}(\mathbb C)^{(2)}$ as a $\hat{\mathfrak g}_0(\mathbb C)$-module.  
By Proposition \ref{prop22} (2) and Remark \ref{rem21}, 
we obtain that $\mathfrak g(\mathbb C)
=\bigoplus\limits_{p\in\mathbb Z}\mathfrak g_p(\mathbb C)$ 
is of type $(A_{l},\{\alpha_1,\alpha_l\})$, $(B_{l},\{\alpha_1,\alpha_l\})$ or 
$(G_{2},\{\alpha_1,\alpha_2\})$. 
Hence the assertions (1)--(3) follow from Proposition \ref{prop24}, 
and the tables of \cite[pp.79--82]{0ni04:0} and \cite[pp.30--32]{war72:0}. 
\end{proof}

\begin{proposition}
Let $(\mathfrak m,[g_1])$ and $(\mathfrak m,[g_2])$ be two conformal pseudo-subriemannian FGLAs of semisimple type. If the prolongation of 
$(\mathfrak m,[g_1])$ coincides with that of $(\mathfrak m,[g_2])$, 
then $(\mathfrak m,[g_1])$ is equivalent to $(\mathfrak m,[g_2])$. 
\label{prop52}
\end{proposition}
\begin{proof} 
The mapping $\varphi=g_1^\sharp\comp g_2^\flat$ induces an isomorphism of $\mathfrak g_{-1}(\mathbb C)$ onto itself as a 
$\hat{\mathfrak g}_0(\mathbb C)$-module. 
If the $\mathfrak g_{0}(\mathbb C)$-module $\mathfrak g_{-1}(\mathbb C)$ is 
reducible, 
then $\mathfrak g_{-1}(\mathbb C)$ is the direct sum of two irreducible 
$\mathfrak g_{0}(\mathbb C)$-modules $\mathfrak g_{-1}(\mathbb C)^{(i)}$ 
$(i=1,2)$ and $\mathfrak g_{-1}(\mathbb C)^{(1)}$ is not isomorphic to  
$\mathfrak g_{-1}(\mathbb C)^{(2)}$ as a $\mathfrak g_0(\mathbb C)$-module. 
In this case $\varphi(\mathfrak g_{-1}(\mathbb C)^{(i)})=\mathfrak g_{-1}(\mathbb C)^{(i)}$ $(i=1,2)$. 
By Schur's lemma, there exist two complex numbers $\lambda_1,\lambda_2$ such that $\varphi|\mathfrak g_{-1}(\mathbb C)^{(i)}=\lambda_i id$ $(i=1,2)$. 
For $X\in \mathfrak g_{-1}(\mathbb C)^{(1)}$ and 
$Y\in \mathfrak g_{-1}(\mathbb C)^{(2)}$ we obtain 
$\lambda_1g_1(X,Y)=g_1(\varphi(X),Y)=g_2(X,Y)$ and 
$\lambda_2g_1(X,Y)=g_1(X,\varphi(Y))=g_2(X,Y)$. 
Since $\mathfrak g_{-1}(\mathbb C)^{(i)}$ are totally isotropic with respect to $g_1$ and $g_2$, 
we get $\lambda_1=\lambda_2$. 
Hence $g_2=\lambda_1g_1$ and $\lambda_i\in\mathbb R$. 
Thus we see that $[g_1]=[g_2]$ or $[g_1]=[-g_2]$. 
Similarly we can prove that $[g_1]=[g_2]$ or $[g_1]=[-g_2]$ when 
the $\mathfrak g_{0}(\mathbb C)$-module $\mathfrak g_{-1}(\mathbb C)$ is 
irreducible. 
\end{proof}
From Theorem \ref{thm51}, Proposition \ref{prop52} and the results of \S4 
we obtain the following theorem. 
\begin{theorem}
Let $(\mathfrak m,[g])$ be a conformal pseudo-subriemannian FGLA of semisimple type. 
Then $(\mathfrak m,[g])$ is equivalent to one of conformal pseudo-subriemannian FGLAs of types 
$({\rm H}\mathbb C)_{p,q}$, $({\rm H}\mathbb C')_{p,q}$, 
$({\rm H}\mathbb H)_{p,q}$, $({\rm H}\mathbb H')_{p,q}$, 
$({\rm H}\mathbb O)$, $({\rm H}\mathbb O')$, $({\rm BI})_l$, $({\rm G})$. 
The prolongation $\gla g$ of $(\mathfrak m,[g])$ and the signature $(r,s)$ of $g$ are given in the following table. 
\label{thm52}
\end{theorem} 
\begin{tabular}{|c|p{2.8cm}|p{2.5cm}|p{2.5cm}|p{2.5cm}|}\hline
{\scriptsize $(\mathfrak m,[g])$} 
& {\scriptsize $({\rm H}\mathbb C)_{p,q}$} {\scriptsize $(p\geqq2$, $q\geqq1)$}
& {\scriptsize $({\rm H}\mathbb C)_{p,0}$} {\scriptsize $(p\geqq2)$}
& {\scriptsize $({\rm H}\mathbb C)_{1,q}$} {\scriptsize $(q\geqq1)$}
& {\scriptsize $({\rm H}\mathbb C')_{p,q}$} {\scriptsize $(p\geqq1$, $2p+q\geqq3)$}
 \\ \hline
{\scriptsize $\gla g$} 
& {\scriptsize $(({\rm AIIIa})_{l,p},\{\alpha_1,\alpha_l\})$ $(l=2p+q-1)$}  
& {\scriptsize $(({\rm AIIIb})_{l},\{\alpha_1,\alpha_l\})$ $(l=2p-1)$} 
& {\scriptsize $(({\rm AIV})_{l},\{\alpha_1,\alpha_l\})$ $(l=q+1)$}
& {\scriptsize $(({\rm AI})_{l},\{\alpha_1,\alpha_l\})$ $(l=2p+q-1)$}
 \\ \hline
{\scriptsize $(r,s)$} 
& {\scriptsize $(2p+2q-2,2p-2)$}
& {\scriptsize $(2p-2,2p-2)$}
& {\scriptsize $(2q,0)$}
& {\scriptsize $(2p+q-2,2p+q-2)$}
 \\ \hline\hline

{\scriptsize $(\mathfrak m,[g])$} 
& {\scriptsize $({\rm H}\mathbb H)_{p,q}$} {\scriptsize $(p,q\geqq1)$}
& {\scriptsize $({\rm H}\mathbb H)_{p,0}$} {\scriptsize $(p\geqq2)$}
& {\scriptsize $({\rm H}\mathbb H')_{p,q}$} {\scriptsize $(p\geqq1$, $2p+q\geqq3)$}
& {\scriptsize $({\rm H}\mathbb O)$} 
\\ \hline 
{\scriptsize $\gla g$} 
& {\scriptsize $(({\rm CIIa})_{l,p},\{\alpha_2\})$ 
$(l=2p+q)$} 
& {\scriptsize $(({\rm CIIb})_{l},\{\alpha_2\})$ $(l=2p)$} 
& {\scriptsize $(({\rm CI})_{l},\{\alpha_2\})$ $(l=2p+q)$}  
& {\scriptsize $({\rm FII},\{\alpha_{4}\})$} 
 \\ \hline

{\scriptsize $(r,s)$} 
& {\scriptsize $(4p+4q-4,4p-4)$}
& {\scriptsize $(4p-4,4p-4)$}
& {\scriptsize $(4p+2q-4,4p+2q-4)$}
& {\scriptsize $(8,0)$} 
 \\ \hline \hline
{\scriptsize $(\mathfrak m,[g])$} 
&  {\scriptsize $({\rm H}\mathbb O')$} 
& {\scriptsize $({\rm BI})_l$} {\scriptsize $(l\geqq2)$}
&  {\scriptsize $({\rm G})$} 
& 
\\ \hline
{\scriptsize $\gla g$}
& {\scriptsize $({\rm FI},\{\alpha_{4}\})$}
& {\scriptsize $(({\rm BI})_{l,l},\{\alpha_1,\alpha_l\})$} 
& {\scriptsize $({\rm G_{2(2)}},\{\alpha_1,\alpha_2\})$} 
& 
\\ \hline
{\scriptsize $(r,s)$} 
& {\scriptsize $(4,4)$}
& {\scriptsize $(l-1,l-1)$}
& {\scriptsize $(1,1)$} 
& 
 \\ \hline
\end{tabular}

\begin{corollary}
Let $(\mathfrak m,[g])$ be a conformal pseudo-subriemannian FGLA of semisimple type. 
Unless $(\mathfrak m,[g])$ is equivalent to one of 
$({\rm H}\mathbb C)_{p,q}$ $(p\geqq2$, $q\geqq1)$, 
$({\rm H}\mathbb C)_{1,q}$ $(q\geqq1)$, 
$({\rm H}\mathbb H)_{p,q}$ $(p,q\geqq1)$, 
$({\rm H}\mathbb O)$, it is conformal neutral-subriemannian. 
\label{cor51}
\end{corollary}

\begin{example}
Let $\mathfrak l=\mathfrak l_{-1}\oplus \mathfrak l_0\oplus 
\mathfrak l_1$ be a real SGLA 
such that 
$\mathfrak l_{-1}\ne\{0\}$. 
We assume that $\mathfrak l$ is splittable and $\rank \mathfrak l\geqq2$. 
Let $S=\bigoplus\limits_{p<0} S_p$ be a faithful irreducible graded 
$\mathfrak l$-module such that $S$ is isomorphic to $\mathfrak l$ 
as an $\mathfrak l$-module and such that $S_{-1}\ne\{0\}$. 
Let $\mathfrak t$ be the semidirect product of $\mathfrak l$ by $S$. 
Here $S$ considers as a commutative Lie algebra. 
We define a gradation $(\mathfrak t_p)$ of $\mathfrak t$ as follows: 
$$\mathfrak t_p=\mathfrak l_p \quad (p\geqq0), \quad 
\mathfrak t_{-1}=\mathfrak l_{-1}\oplus S_{-1}, \quad 
\mathfrak t_q=S_q\quad  (q\leqq-2).$$
Then $\gla t$ becomes a GLA such that the negative part $\mathfrak m$ is 
an FGLA of the third kind. 
By assumption $\mathfrak l_{-1}$ is contragredient to 
$S_{-1}$ as a $\mathfrak t_{0}$-module. That is, 
there exists a $\mathfrak t_0$-module isomorphism $\varphi$ 
of $\mathfrak l_{-1}$ onto $S_{-1}^*$.  
We define a symmetric bilinear form $g$ on $\mathfrak t_{-1}$ as follows: 
$$g(X,Y)=g(Z,W)=0,\quad g(X,Z)=g(Z,X)=\langle \varphi(X),Z\rangle\quad 
(X,Y\in\mathfrak l_{-1},\ Z,W\in S_{-1}).$$
Then $g$ is nondegenerate, and hence 
$(\mathfrak m,[g])$ becomes a conformal neutral-subriemannian FGLA. Clearly $\mathfrak t$ is contained in the prolongation $\gla g$ of 
$(\mathfrak m,[g])$. 
If $\gla g$ is of type $(({\rm BI})_{l,l},\{\alpha_1,\alpha_l\})$ $(l\geqq3)$, 
then $\{~X\in\mathfrak g_{-1}:[X,\mathfrak g_{-2}]=\{0\}~\}=\{0\}$, which is a contradiction. 
By Theorem \ref{thm52}, $(\mathfrak m,[g])$ is not of semisimple type and 
$\mathfrak g_1\ne\{0\}$. 
\label{exa51}
\end{example}
 

\begin{thebibliography}{[SSY96]}  
\bibitem{amt06:01}
D.~V.~Alekseevsky, C.~Medori, and A.~Tomassini, 
Maximally homogeneous para-CR manifolds, 
Ann. Glob. Anal. Geom. \textbf{30} (2006) 1--27. 

\bibitem{BA02:1} 
U.~Bader and A.~Nevo, 
Conformal actions of semisimple Lie groups on 
compact pseudo-Riemannian manifolds, 
J.~Differential Geometry {\bf 60} (2002) 355--387. 

\bibitem{bou68:1} 
N.~Bourbaki, 
Groupes et alg\`ebres de Lie, Chapitres 4, 5 et 6, 
Hermann Paris, 1968. 

\bibitem{bou75:1} 
N.~Bourbaki, 
Groupes et alg\`ebres de Lie, Chapitres 7 et 8, 
Diffusion CCLS, 1975. 

\bibitem{co15:01}
M.~G.~Cowling and A.~Ottazzi, 
Conformal maps of Carnot groups, 
Ann. Acad.  Sci. Fennica, 
Mathematica 
\textbf{40}, (2015) 203--213.

\bibitem{cs09:0} 
A.~\v{C}ap and J.~Slavak, 
Parabolic geometries I, background and general theory, 
AMS, 2009.

\bibitem{doj82:0} 
D.~Djokovic, 
Classification of $\mathbb Z$-graded real semisimple Lie algebras, 
J. of Algebra \textbf{76}, (1982) 367--382. 

\bibitem{gom96:0} 
S.~Gomyo, 
Realization of the exceptional simple graded Lie algebras of the second kind, 
Algebras, groups and geometries \textbf{13}, (1996) 431--464.


\bibitem{gov94:0}
V.~V.~Gorbatsevic, A. L.~Onishchik, and E.~B.~Vinberg, 
Structure of Lie groups and Lie algebras, 
in: Lie Groups and Lie Algebras III, Encyclopedia Math. Sci. \textbf{41}, 
Springer-Verlag 
Berlin, 1994. 

\bibitem{kob72:1} 
S.~ Kobayashi, 
Transformation groups in differential geometry, 
Springer, 1972. 

\bibitem{0ni04:0}
A.~L.~Onishchik, 
Lectures on real semisimple Lie algebras and their representations, 
EMS, 2004. 

\bibitem{tak65:01}
M.~Takeuchi, 
Cell decompositions and Morse equalities on certain symmetric spaces, 
J. Fac. Sci. Univ. Tokyo, \textbf{12} (1965) 81--192. 

\bibitem{tan70:1}
N.~Tanaka,
On differential systems, graded Lie algebras and pseudo-groups, 
J. Math. Kyoto Univ. \textbf{10}, (1970) 1--82. 

\bibitem{war72:0} 
G.~Warner, 
Harmonic analysis on semi-simple Lie groups, I,
Springer-Verlag, New York, 1972.

\bibitem{yam93:1} 
K.~Yamaguchi, 
Differential systems associated with simple graded Lie algebras, 
Advanced Studies in Pure Math. \textbf{22} (1993) 413--494. 


\bibitem{yat15:01} 
T.~Yatsui, 
On conformal subriemannian fundamental graded Lie algebras
and Cartan connections,
Lobachevskii J. of Math. \textbf{36} (2015) 170--178.

\end{thebibliography}
\end{document}